\documentclass[11pt]{article}
\usepackage{amsmath, amssymb, amsfonts, amstext, amsthm, textcomp, enumerate}
\usepackage[mathscr]{euscript}
\usepackage{float}
\usepackage{booktabs}
\usepackage{mathtools}
\usepackage[left=20mm,top=0.5in,bottom=15mm]{geometry}
\usepackage{graphicx}
\usepackage{caption}
\usepackage{epstopdf}
\usepackage{longtable}
\usepackage[utf8]{inputenc}
\usepackage{color}
\usepackage{xcolor}
\usepackage{hyperref}
\usepackage{graphicx}
\usepackage{dcolumn}% Align table columns on decimal point
\usepackage{bm}% bold math
\usepackage{epstopdf}
\usepackage[english]{babel}
\usepackage{subfigure}
\usepackage{color}
\usepackage{ulem}

\newtheorem{thm}{Theorem}[section]
\newtheorem{lem}[thm]{Lemma}%[section]

\newtheorem{cor}[thm]{Corollary}%[section]
\newtheorem{pro}[thm]{Proposition}%[section]
\newtheorem{defn}[thm]{Definition}%[section]

\newtheorem{ex}[thm]{Example}%[section]

\newcommand{\R}{\mathbb{R}}
\newcommand{\stmaj}{\prec^{*}}
\newcommand{\ds}{\downarrow}
\newcommand{\us}{\uparrow}

\newfont{\bb}{msbm10}

\begin{document}
\allowdisplaybreaks
	\title{Variations on Majorization of Vectors and Connections to Determinantal Inequalities}
    \author{Shaun Fallat\thanks{Department of Mathematics and Statistics,
University of Regina, Regina, SK, Canada 
(sfallat@uregina.ca).}
    \and  Samir Mondal\thanks{Department of Mathematics and Statistics,
University of Regina, Regina, SK, Canada (isamirmondal@gmail.com)}
  \and Hristo Sendov\thanks{Department of Statistical and Actuarial Sciences, Department of Mathematics, The University of Western Ontario, London, Ontario, Canada (hsendov@uwo.ca)} 
  }
 \maketitle

%\begin{abstract}
%Comparing vector quantities often leads to inequalities expressible via majorization, which, in turn, are closely connected to doubly stochastic matrices.  Majorization relations between the eigenvalue 
%$n$-tuples of matrices, particularly in the context of Hadamard’s and Fischer’s inequalities, are studied in detail in  [R. Bhatia. {\it Matrix analysis}, New York: Springer-Verlag; 1997.] In this work, we focus on exploring the majorization relations between the eigenvalue 
%n$-tuples, with special attention to the determinantal inequalities of Koteljanskii and Szasz.
%In addition, we develop natural connections to blocks of doubly stochastic matrices and $T$-transforms.
%\end{abstract}

\begin{abstract}
Majorization is a fundamental tool for comparing vectors, with connections to convexity, doubly stochastic matrices, eigenvalues, singular values, and zeros of polynomials. In matrix analysis, it plays a central role in the study of eigenvalue inequalities, particularly those arising from classical determinantal inequalities such as those attributed to Hadamard and Fischer in the context of positive semidefinite matrices. A result of Fischer and Holbrook shows that equality in the Hardy--Littlewood--P\'olya theorem for non-affine convex functions is closely linked to block structure in the associated doubly stochastic transformations.

Motivated by this, we introduce $*$-majorization, a structured extension of majorization that respects prescribed block decompositions of vectors. This framework naturally corresponds to block diagonal doubly stochastic matrices and provides a refinement of the classical Hardy--Littlewood--P\'olya and Rado theorem. We show that such transformations are precisely the linear operators that preserve $*$-majorization, and we extend fundamental constructions such as $T$-transforms and convex combinations to this setting.

In an application, we study the eigenvalue relations associated with the principal submatrices of positive definite matrices. Classical majorization does not, in general, capture determinantal inequalities such as those of Koteljanskii, whereas $*$-majorization provides a natural framework for structured comparisons of eigenvalue vectors. This leads to new insights into the interplay between majorization theory, determinantal inequalities, and spectral properties of matrices.
\end{abstract}

% edit keywords and MSC as needed
\noindent Keywords: eigenvalues, majorization, doubly stochastic matrices, determinantal inequalities, positive definite matrices.  

\noindent AMS-MSC: 15A42, 15A18; 15B48

 % $x \prec y$
 % $x^{\downarrow}$
 % $x^{\uparrow}$

\vspace{.5cm}
\section{Introduction and Background}
 The theory of majorization, which originated with the work of Hardy, Littlewood, and P\'olya~\cite{Hardy1929} and was later extended by Rado~\cite{Rado1952}, provides a framework for comparing vectors in $\R^n$. It plays a central role in matrix analysis, convexity, and inequalities, with applications ranging from operator theory to probability, statistics, combinatorics, and quantum mechanics; see \cite{Bhatia1997, Marshall2011, Niculescu2006, Nielsen2000, Phelps2001}. Its central result characterizes the pre-order $y \prec x$ in terms of convex inequalities and doubly stochastic transformations. This framework extends naturally to function spaces, where vectors are replaced by integrable functions and matrices by doubly stochastic operators; see \cite{Ryff1965}.

Majorization also provides a natural framework for comparing the disorder of probability distributions and extends to quantum mechanics through density matrices. For $n$-dimensional Hermitian matrices $R$ and $S$, one writes $R \prec S$ if the vector of eigenvalues of $R$ is majorized by that of $S$. In this way, majorization characterizes the ordering of quantum states via their spectra and plays a central role in quantum information theory. It is also fundamental in matrix analysis, where it underlies eigenvalue inequalities,  such as those of Hadamard and Fischer for positive semidefinite matrices and Fan and Lidskii for any Hermitian matrices. 

For any $x \in \mathbb{R}^n$, let $x^{\downarrow}$ and
 $x^{\uparrow}$ denote the vectors
obtained from $x$ by rearranging the coordinates in non-increasing or 
increasing order, respectively. That is, $x^{\downarrow}_1\geq x^{\downarrow}_2\geq \dots \geq x^{\downarrow}_n$ and  $x^{\uparrow}_1\leq x^{\uparrow}_2\leq \dots \leq x^{\uparrow}_n.$ Clearly, we have
$x^{\downarrow}_i = x^{\uparrow}_{n-i+1} \text{ for all } i=1,\ldots, n.$

For $x, y \in \mathbb{R}^n$, we say that { is $x$ weakly majorized by $y$}, denoted as $x \prec_w y$, if \begin{equation*}
\label{weakmajor}
    \sum_{i=1}^{k}x^{\downarrow}_i\leq \sum_{i=1}^{k}y^{\downarrow}_i 
   \end{equation*} 
   for $1\leq k\leq n$. If, in addition, the inequality for $k=n$ holds with equality, then we say that {$x$ is majorized by $y$}, denoted as $x \prec y$.

A linear map $T:\R^n\to\R^n$ is called a {\it $T$-transform} if there exist indices $j\neq k$ and a scalar $t\in[0,1]$, such that
\[
Ty=(y_1,\dots,y_{j-1},ty_j+(1-t)y_k,y_{j+1},\dots, y_{k-1},(1-t)y_j+ty_k,y_{k+1},\dots,y_n).
\]

We recall the following classical theorem, in a form suitable for our purposes, due to G. H. Hardy, J. E. Littlewood, George Pólya \cite{Hardy1929, HardyLittlewoodPolya1952}, and Richard Rado \cite{Rado1952}. It reveals the connection between majorization, convexity, symmetry, and linear operators.
\begin{thm}
\label{thm:HLP}
Let $x,y \in \R^n$. The following statements are equivalent:
\begin{enumerate}
\item We have $x \prec y$.
\item The vector $x$ is obtained from $y$ by a finite sequence of $T$-transforms.
\item For every continuous convex function $f:\R\to\R$,
\begin{align}
\label{2026-04-19-conv}
\sum_{i=1}^n f(y_i) \le \sum_{i=1}^n f(x_i).
\end{align}
\item There exists a doubly stochastic matrix $M\in\R^{n\times n}$, such that
$y = Mx.$
\item The vector $y$ lies in the convex hull of all permutations of $x$.
\end{enumerate}
\end{thm}

Fischer and Holbrook~\cite{FischerHolbrook1977}  performed a detailed analysis of when equality holds in \eqref{2026-04-19-conv}. They introduced the notion of {\it doubly stochastic by blocks} matrices. Their work showed that for non-affine convex functions, equality in \eqref{2026-04-19-conv}
is closely related to the presence of a block structure in the corresponding doubly stochastic transformation. In such cases, the transformation mixes coordinates only within prescribed blocks, rather than globally across the entire vector.
This motivates the introduction of a refined notion of majorization adapted to a prescribed block decomposition.

% \begin{defn}[$*$-majorization]
% \label{def:star-majorization}
% Let $x\in\R^n$, $y\in\R^m$, $z\in\R^r$, and $w\in\R^s$ with
% \[
% n+m=r+s.
% \]
% Set
% \[
% u=(x,y)\in\R^{n+m},
% \qquad
% v=(z,w)\in\R^{r+s}.
% \]
% We say that $u$ is {\it $*$-majorized} by $v$, and write
% \[
% u \stmaj v,
% \]
% if
% \[
% \sum_{i=1}^k u_i \le \sum_{i=1}^k v_i,
% \qquad k=1,\dots,n+m,
% \]
% with equality when $k=n+m$.
% \end{defn}

\begin{defn}[$*$-majorization]
\label{$*$-majorization}
Let $u=(x,y)\in\mathbb{R}^{n+m}$ and $v=(z,w)\in\mathbb{R}^{r+s}$ with $n+m=r+s$.  
Let
\[
u^*=(x^{\sigma_1},y^{\sigma_2}), \qquad 
v^*=(z^{\tau_1},w^{\tau_2}),
\]
where $\sigma_1,\sigma_2,\tau_1,\tau_2 \in \{\uparrow,\downarrow\}$.
We say that $u$ is $*$-majorized by $v$, denoted $u \stmaj v$, if
\[
\sum_{i=1}^{k} u_i^* \le \sum_{i=1}^{k} v_i^*, \,\,\, \mbox{ for } k=1,\dots,n+m,
\]
where the last inequality holds with equality. 
\end{defn}

We may consider alternative orderings depending on the application, such as
$(x^\ds,y^\ds)$, $(x^\us,y^\us)$,$(x^\ds,y^\us)$, $(x^\us,y^\ds),$
and similarly for $(z,w)$.

It should be clear that when $n=r$ and $m=s$, we have that if $x \prec z$ and $y \prec w$, then $(x^{\downarrow}, y^{\downarrow}) \prec^{*} (z^{\downarrow}, w^{\downarrow})$. Conversely, if $(x^{\downarrow}, y^{\downarrow}) \prec^{*} (z^{\downarrow}, w^{\downarrow})$, then $x \prec_w z $ and $-y \prec_w -w$. In addition, $(x^{\downarrow}, 0) \prec^{*} (z^{\downarrow}, 0)$ or $(0, y^{\downarrow}) \prec^{*} (0, w^{\downarrow})$ imply that $x \prec z$ or $y \prec w$, respectively.

The relation $\stmaj$ defines a pre-order that reflects a restricted form of mixing of the coordinates. In contrast with classical majorization, no global rearrangement is performed; instead, the prescribed ordering is retained, and comparisons are made directly on the concatenated vectors.

% This notion is naturally associated with block diagonal doubly stochastic matrices. Indeed, while classical majorization corresponds to the action of all doubly stochastic matrices, the relation $\stmaj$ corresponds to transformations of the form
% \[
% A = A_1 \oplus A_2,
% \]
% where $A_1\in\R^{n\times n}$ and $A_2\in\R^{m\times m}$ are doubly stochastic. Such operators preserve the block decomposition and allow mixing only within each block.

% The main purpose of this paper is to make this correspondence precise. We show that doubly stochastic block operators are exactly those linear transformations that preserve $\stmaj$ in an appropriate sense. This result may be viewed as a block analog of Theorem~\ref{thm:HLP}.

% We also extend $*$-majorization to finitely many blocks and show that it corresponds to direct sums of doubly stochastic matrices. This leads to a structured version of majorization theory that captures block-preserving transformations and provides a natural framework for studying equality phenomena in convex inequalities.

This notion is naturally associated with {block diagonal doubly stochastic matrices}. In the classical setting, majorization corresponds to the action of all doubly stochastic matrices. In contrast, the relation $\stmaj$ is related to a more structured class of transformations, namely those of the form
\[
A = A_1 \oplus A_2,
\]
where $A_1 \in \mathbb{R}^{n \times n}$ and $A_2 \in \mathbb{R}^{m \times m}$ are doubly stochastic matrices. Along these lines we prove this interesting analog from classical majorization theory.

\begin{thm}
The matrix $ A \in \mathbb{R}^{(n+m)\times (n+m)} $ is a doubly stochastic  block matrix if and only if we have
\begin{align}
\label{2025-11-24-*maj}
(Au)^\natural \stmaj u^\natural, \,\,\, \text{ for all } u \in  \mathbb{R}^n \times \mathbb{R}^m,
\end{align}
where $u^\natural$ denotes the non-increasing  block order, that is, if $u = (x,y)$, then $u^\natural =(x^\ds,y^\ds)$.
\end{thm}

The main objective of this paper is to make this correspondence precise. In particular, we show that block doubly stochastic matrices can be characterized by their preservation of $\stmaj$, thus providing a natural block analog of the classical Hardy--Littlewood--P\'olya--Rado theorem.

Within this framework, several classical results admit natural extensions. In particular:
\begin{itemize}
    \item Block doubly stochastic matrices are precisely those linear operators that preserve $*$-majorization;
    
    \item Classical constructions such as $T$-transforms, permutations, and convex combinations admit natural block counterparts that act within individual blocks.
\end{itemize}

% Furthermore, we extend $*$-majorization to the setting of finitely many blocks and show that it corresponds to direct sums of doubly stochastic matrices. This leads to a structured version of majorization theory that captures block-preserving transformations and provides a natural framework for studying equality phenomena in convex inequalities.

Beyond its theoretical significance, the block perspective developed in this work has important applications, particularly in the study of positive definite matrices and their principal submatrices. Let 
\(
N=\{1,2,\dots,n\},
\)
and let $A$ be an $n \times n$ positive definite (Hermitian) matrix. For any subset $\alpha \subseteq N$, we denote by $A[\alpha]$ the principal submatrix of $A$ with rows and columns indexed by $\alpha$, and by $\lambda(A)$ the eigenvalue vector of $A$.

% Classical matrix analysis provides a rich collection of determinantal inequalities relating such principal submatrices, including Hadamard’s, Fischer’s, and Koteljanskii’s inequalities. These results describe multiplicative relationships between determinants and, in the case of positive semidefinite matrices, are known to be closely related and, in fact, equivalent. A natural question is whether analogous relations can be formulated at the level of eigenvalues, thereby providing a spectral counterpart to these classical inequalities.

% We now turn to applications involving positive definite matrices. Let 
% \(N=\{1,2,\dots,n\}\), and let \(A=(a_{ij})\) be a \(n\times n\) complex matrix. 
% For \(\alpha \subseteq N\), denote by \(A[\alpha]\) the principal submatrix 
% of \(A\) indexed by \(\alpha\), with the convention \(A[\emptyset]=1\). 
% For a Hermitian matrix \(A\), let \(\lambda(A)\) denote its eigenvalue vector. 
When \(A\) is positive semidefinite, there is a rich variety of inequalities 
relating its principal minors and eigenvalues. We recall several classical 
determinantal inequalities that hold for positive definite matrices (see \cite{Had, Fis, Ko1,Ko2}).

\begin{itemize}
    \item Hadamard's inequality: 
    \begin{equation}
    \label{2025-11-20-Had}
        {\rm det}(A)\leq a_{11}\cdots a_{nn}.
    \end{equation}
    \item Fisher's inequality: for any $\alpha \subseteq N$, one has
    \begin{equation}
        {\rm det}(A)\leq {\rm det}(A[\alpha]) {\rm det}(A[\alpha^c]),
    \end{equation}
    where $\alpha^c$ is the complement of $\alpha$ in $N.$
    \item Koteljanskii's inequality: for any  $\alpha,\beta\subseteq N,$ one has
    \begin{equation}
        {\rm det}(A[\alpha \cup \beta]) {\rm det}(A[\alpha \cap \beta])\leq {\rm det}(A[\alpha]) {\rm det}(A[\beta]).
    \end{equation}
    
    \item Szasz’s inequality:
    \begin{equation}
        P_{k+1}(A)^{{\binom{n-1}{k}}^{-1}} \leq P_{k}(A)^{{\binom{n-1}{k-1}}^{-1}} \;\;\text{for each} \;\;k=1, \dots , n,
     \end{equation}
     where $P_k(A)$ denotes the product of all  principal minors of $A$ of order $k$.
    
\end{itemize}

These inequalities for positive semidefinite matrices are known to be equivalent.

Hadamard's determinatal inequality for positive definite matrices \eqref{2025-11-20-Had}, is a classical result that can be seen as a consequence of Schur's theorem, 
\begin{equation}
\label{Schur} 
{\rm diag}(A)\prec \lambda(A) 
\end{equation} combined with the well-known fact that the negative of any elementary symmetric function  is Schur convex, see \cite[Example II.3.16]{Bhatia1997}, on the non-negative orthant. Similarly, Fisher's inequality is a consequence of the majorization relation
 \begin{equation}\label{Frisher's} 
(\lambda(A[\alpha]), \lambda(A[\alpha^c]))\prec \lambda(A),
\end{equation} 
which follows from a majorization result related to `pinching' of a matrix, see \cite[Problem II.5.5]{Bhatia1997}. It is worth noting that Szasz’s inequalities can also be derived from suitable majorization relations between collections of eigenvalues of principal submatrices; see \cite{SendovYuan2026} for a recent approach based on spectral hierarchy.

A natural question is whether Koteljanskii’s inequality is also a consequence of certain majorization relations. In particular, one might ask whether 
\begin{equation}
\label{Kotlenjenskii} 
(\lambda(A[\alpha]), \lambda(A[\beta]))\prec (\lambda(A[\alpha\cap \beta]), \lambda(A[\alpha\cup \beta]))\end{equation} holds in general. Unfortunately, \eqref{Kotlenjenskii} is not true in general, as demonstrated in the next example.

\begin{ex}
 Consider the positive definite matrix 
 $$
 A=\begin{pmatrix}
     5 & 1 & 1 & 1\\
     1 & 4 & 1 & 1\\
     1 & 1 & 2 & 1 \\
     1 & 1 & 1 & 3
 \end{pmatrix},
 $$ 
 with eigenvalues $\lambda(A)=\{6.8039, 3.5077, 2.3923, 1.2961\}.$  Let $\alpha=\{1,2, 3\}$ and $\beta=\{2, 3, 4\}.$ Then the eigenvalues of the corresponding principal submatrices are:
 \begin{itemize}
     \item $\lambda(A[\alpha])=\{6.0861, 3.4280, 1.4859\}$
     \item $\lambda(A[\beta])=\{5.2143, 2.4608, 1.3249\}$, and 
     \item $\lambda(A[\alpha \cap \beta])=\{4.4142, 1.5858\}.$
 \end{itemize}
%Now define $$x=(\lambda(A[\alpha]), \lambda(A[\beta]))\;\;\; \text{and}
%\;\;\; y=(\lambda(A[\alpha\cap \beta]), \lambda(A[\alpha\cup \beta])).$$ 
If the majorization relation \eqref{Kotlenjenskii} were to hold, then the partial sums of the non-increasing rearrangements would satisfy (\ref{weakmajor}) for all, $1\leq k\leq n.$ However, this fails for $k=2$. 
\end{ex}
We show that, although classical majorization fails, the eigenvalue vectors associated with the principal submatrices above satisfy several $*$-majorization relations, thus providing a natural majorization-type interpretation of Koteljanskii’s inequality. As an application, we establish new connections between eigenvalue majorization, Ky Fan–type variational inequalities, and Koteljanskii’s determinantal inequality. In particular, we establish the next results concerning $*$-majorization in connection to Koteljanskii’s classical determinantal inequality.

\begin{thm}
\label{2025-11-24-thm}
Let $A$ be a positive definite matrix and $\alpha,\beta \subseteq N$.
Let
\begin{align*}
 x = \lambda(A[\alpha]) \in \mathbb{R}^n, \quad 
 y = \lambda(A[\beta]) \in \mathbb{R}^m, \quad 
 z = \lambda(A[\alpha \cup \beta]) \in \mathbb{R}^r, \quad 
 w = \lambda(A[\alpha \cap \beta]) \in \mathbb{R}^s.
\end{align*}
Then the following relationships hold:
\begin{enumerate}[(1)]
 \item \label{2026-01-27-maj1} $(x^{\downarrow},y^{\uparrow})\prec^*(z^{\downarrow},w^{\downarrow})$,
 \item \label{2026-01-27-maj2} $(x^{\downarrow},y^{\uparrow})\prec^*(z^{\downarrow},w^{\uparrow})$,
 \item \label{2026-01-27-maj3} $(x^{\downarrow},y^{\downarrow})\prec^*(z,w)^{\downarrow}$.
\end{enumerate}
\end{thm}

The remainder of this paper is divided intto two parts. In this next section, we introduce the notion of $*$-majorization as a block-structured extension of classical majorization and provide a characterization of block doubly stochastic matrices as precisely those linear operators that preserve $*$-majorization. In the final section, 
 we establish new spectral $*$-majorization relations for the eigenvalues of the principal submatrices of positive definite matrices, providing a majorization-theoretic interpretation of Koteljanskii’s classical determinantal inequality.

\section{Block structure and $*$-majorization}

In this section, we explain how $*$-majorization is related to ordinary majorization.

\begin{pro}
\label{prop:block-from-classical}
Let $x,z\in\R^n$ and $y,w\in\R^m$. If
$x\prec z$ and $y\prec w$, then
\[
(x^\ds,y^\ds)\stmaj (z^\ds,w^\ds).
\]
\end{pro}

The converse of Proposition~\ref{prop:block-from-classical} does not hold unless one interprets the inequalities in a block way. The first $n$ inequalities encode information about $x$ and $z$, while the remaining inequalities encode information concerning $y$ and $w$ once the first respective block totals are added. 

The matrix class in the block setting is the direct sum of doubly stochastic matrices. A {\it block permutation matrix} is a matrix of the form
\[
P=P_1\oplus P_2,
\]
where $P_1$ is an $n\times n$ permutation matrix and $P_2$ is an $m\times m$ permutation matrix.
It is clear that the convex hull of all block permutation matrices is precisely the set of all block doubly stochastic matrices.

 It is well known that doubly stochastic matrices can be characterized in terms of majorization (see \cite{Marshall2011}).  

\begin{thm} 
\label{thm:classical-ds}
    A matrix $ A \in \mathbb{R}^{n \times n} $ is doubly stochastic if and only if
$Ax \prec x$ for all $x \in \mathbb{R}^n$.
\end{thm}

Similarly, doubly stochastic block matrices can be characterized in terms of $*$-majorization.

\begin{thm}\label{thm:block-ds-star}
The matrix $ A \in \mathbb{R}^{(n+m)\times (n+m)} $ is a doubly stochastic  block matrix if and only if we have
\begin{align}
\label{2025-11-24-*maj}
(Au)^\natural \stmaj u^\natural, \,\,\, \text{ for all } u \in  \mathbb{R}^n \times \mathbb{R}^m,
\end{align}
where $u^\natural$ denotes the non-increasing  block order, that is, if $u = (x,y)$, then $u^\natural =(x^\ds,y^\ds)$.
\end{thm}

\begin{proof}
($\Rightarrow$) Suppose $ A = A_{11} \oplus A_{22} $, where $ A_{11} \in \mathbb{R}^{n \times n} $ and $ A_{22} \in \mathbb{R}^{m \times m} $ are both doubly stochastic. Then for any $ (x, y) \in \mathbb{R}^n \times \mathbb{R}^m $, we have
$A(x, y)^T = \left( A_{11} x, A_{22} y \right)^T.$
By Theorem~\ref{thm:classical-ds}, we have
$A_{11} x \prec x$ and  $A_{22} y \prec y,$ and thus the result follows. 

($\Leftarrow$) Now assume that \eqref{2025-11-24-*maj} holds.  We will show that $ A = A_{11} \oplus A_{22} $, where both $ A_{11} $ and $ A_{22} $ are doubly stochastic. Write $ A $ in block form
$$
A = \begin{bmatrix}
A_{11} & A_{12} \\
A_{21} & A_{22}
\end{bmatrix}, \quad \text{with } A_{ij} \in \mathbb{R}^{n_i \times n_j},\; n_1 = n,\; n_2 = m.
$$
\textbf{Step 1}:  show that $A_{12}=0$, $A_{21}=0$, $A_{11} \ge 0$, and $A_{22} \ge 0$. Moreover, we show that $A_{11}=1$ or $A_{22}=1$ whenever $n=1$ or $m=1$, respectively. %\footnote{Rewrote the proof so that it considers the cases $(x, y) = (0, e_j), (e_i,0), (0, -e_j), (-e_i,0)$ one after another. This seems more structured to me. In addition, for the current step 2 we need that $A_{12}=0$ and $A_{21}=0$. We cannot conclude that $A_{ii}$ is doubly stochastic by considering $A^T$ as it was before.}

Consider the standard basis vector $(x, y) = (0, e_j)$,  then \eqref{2025-11-24-*maj} gives
$$
((A_{12} e_j)^\downarrow, (A_{22} e_j)^\downarrow) \prec^* (0, e_j^\downarrow).
$$
We use  the comments after Definition~\ref{$*$-majorization} repeatedly in the proof. This implies that $A_{12} e_j \prec_w 0$ and $-A_{22} e_j \prec_w - e_j$ for every $j$. This implies that $A_{12} \leq 0 $ and $A_{22} \ge 0$ when $m \ge 2$ or $A_{22} \ge 1$ when $m=1$.

Consider the basis vector $ (x, y) = (e_i, 0) $ in \eqref{2025-11-24-*maj}, which implies
$$
((A_{11} e_i)^\downarrow, (A_{21} e_i)^\downarrow) \prec^* (e_i^\downarrow, 0).
$$
This implies $A_{11} e_i \prec_w e_i$ and $-A_{21} e_i \prec_w 0$ for all $i$. Thus $A_{11} \le 1$ when $n=1$ and  $A_{21} \ge 0.$

%
%
%Then
%$$
%A(x, y)^T = \begin{bmatrix}
%A_{11} e_i \\
%A_{21} e_i
%\end{bmatrix} \prec^* \begin{bmatrix}
%e_i^{(n)} \\
%0
%\end{bmatrix}.
%$$
%This gives $ A_{11} e_i \prec_w e_i $, and from the trace condition, we conclude $ A_{21} \geq 0 $.

Next, taking the vector $(x, y) = (0, -e_j)$ in \eqref{2025-11-24-*maj}, gives
$$
((-A_{12} e_j)^\downarrow, (-A_{22} e_j)^\downarrow) \prec^* (0, (-e_j)^\downarrow).
$$
This implies $-A_{12} e_j \prec_w 0$ and $A_{22} e_j \prec_w e_j$ for all $j$. Thus, $A_{12} \ge 0 $ and $A_{22} \le 1$ when $m=1$.

Finally, taking $(x, y) = (-e_i, 0) $, we have
$$
((- A_{11} e_i)^\downarrow, (- A_{21} e_i)^\downarrow) \prec^* ( (- e_i)^\downarrow, 0).
$$
Thus, $ -A_{11} e_i \prec_w -e_i$ and $A_{21} e_i \prec_w 0$ for all $i$. This implies $A_{11} \ge 1$ when $n=1$ and $A_{21} \le 0.$

Combining all findings gives the claims in Step 1.

\textbf{Step 2}: Show that $A_{11}$ and $A_{22}$ are doubly stochastic matrices. Indeed, referring to \eqref{2025-11-24-*maj} take $ (x, y) = (e, 0)$, where $e \in \mathbb{R}^{n}$ is the all-ones vector, to obtain
$$
((A_{11} e)^\downarrow, 0) \prec^* (e, 0).
$$
The comments after Definition~\ref{$*$-majorization}, implies that $A_{11} e \prec e$ and since $ A_{11} \geq 0 $, we conclude that $ A_{11} e = e $ and hence all row-sums are equal to $1$. To see that the column sums are also $1$, consider $ (x, y) = (e_i, 0)$ for all $i$. Similarly, one can show that $A_{22}$ is also doubly stochastic.
%Noting the trace condition we have $ \sum A_{22} e = m $. Since $ A_{22} \geq 0 $, and $ A_{22} e \prec e $, we conclude $ A_{22} e = e $, so all row sums are 1.  
%Similarly, testing with $ (x, y)^T = (0, e)^T $ gives $ A_{22}^T e = e $, so $ A_{22} $ is also column stochastic. Thus
%$A_{22}$ is doubly stochastic.
%
%\textbf{Conclusion:}
%
%We have shown that:
%$$
%A = \begin{bmatrix}
%A_{11} & 0 \\
%0 & A_{22}
%\end{bmatrix}, \quad A_{11}, A_{22} \geq 0,\quad A_{11} e_n = e_n,\quad A_{22} e_m = e_m,
%$$
%and similarly:
%$$
%A_{11}^T e_n = e_n,\quad A_{22}^T e_m = e_m.
%$$
%Therefore, $ A_{11} $ and $ A_{22} $ are doubly stochastic matrices, and
%$A = A_{11} \oplus A_{22}$, with  $A_{11}, A_{22}$ doubly stochastic.
\end{proof}

We now define the block analogue of the classical $T$-transform.
\begin{defn}[Block $T$-transform]
\label{def:block-T}
A linear map $T:\R^{n+m}\to\R^{n+m}$ is called a {block $T$-transform} if either
\[
T=T_1\oplus I_m \mbox{ or } T=I_n\oplus T_2
\]
where $T_1$ is a classical $T$-transform on $\R^n$ and $T_2$ is a classical $T$-transform on $\R^m$.
\end{defn}
The following is the block analog of the classical $T$-transform.
\begin{lem}
\label{lem:block-T-star}
If $T$ is a block $T$-transform and $u=(x,y)$, then
\[
(Tu)^\natural \stmaj u^\natural,
\]
where $u^\natural$ denotes the non-increasing  block order, that is, $(x^\ds,y^\ds)$.
\end{lem}

\begin{proof}
If $T=T_1\oplus I_m$, then
$Tu=(T_1x,y).$ Since $T_1$ is a classical $T$-transform, $T_1x \prec x.$
By Proposition~\ref{prop:block-from-classical}, we see
\[
((T_1x)^\ds,y^\ds)\stmaj (x^\ds,y^\ds).
\]
The proof for $T=I_n\oplus T_2$ is identical.
\end{proof}

The proof of the following block analog of Theorem~\ref{thm:HLP} is evident.

\begin{thm}[Block analogue of majorization equivalences]
\label{thm:block-equivalences}
Let $x,z\in\R^n$ and $y,w\in\R^m$. Set
\(
u=(x^\ds,y^\ds)\) and \(v=(z^\ds,w^\ds).
\)
Then the following statements are equivalent:
\begin{enumerate}

\item $x\prec z$ and $y\prec w$.
\item $(x,y)$ is obtained from $(z,w)$ by a finite sequence of block $T$-transforms.
\item $(x,y)$ lies in the convex hull of all vectors obtained from $(z,w)$ by permuting coordinates within the first block and within the second block.
\item For some doubly stochastic matrices $A_1\in\R^{n\times n}$ and $A_2\in\R^{m\times m}$, we have
\[
(x,y)^T=(A_1\oplus A_2)(z,w)^T.
\]
\end{enumerate}
\end{thm}

One can extend the notion of $*$-majorization to finitely many blocks in a natural way.

% \begin{defn}[$*$-majorization for multiple blocks]
% Let
% \[
% x^{(j)} \in \R^{n_j}, \qquad y^{(j)} \in \R^{n_j}, \qquad j=1,\dots,m,
% \]
% where \(n_1+\cdots+n_m=N\). Form the concatenated vectors
% \[
% u=(x^{(1)},x^{(2)},\dots,x^{(m)}), \qquad
% v=(y^{(1)},y^{(2)},\dots,y^{(m)}) \in \R^N,
% \]
% where each block is written in a prescribed order. We say that \(u\) is {$*$-majorized} by \(v\), and write
% \[
% u \stmaj v,
% \]
% if
% \[
% \sum_{i=1}^k u_i \le \sum_{i=1}^k v_i, \qquad k=1,\dots,N,
% \]
% with equality for \(k=N\).
% \end{defn}

\begin{defn}[$*$-majorization for multiple blocks]
Let
\[
u=(x^{(1)},\dots,x^{(m)}), \quad 
v=(y^{(1)},\dots,y^{(m)}) \in \mathbb{R}^N,
\]
where $x^{(j)},y^{(j)}\in\mathbb{R}^{n_j}$, $j=1,\dots,m$, and $\sum_{j=1}^m n_j=N$.

Let
\[
u^*=\bigl((x^{(j)})^{\sigma_j}\bigr)_{j=1}^m, \qquad
v^*=\bigl((y^{(j)})^{\tau_j}\bigr)_{j=1}^m,
\]
with $\sigma_j,\tau_j\in\{\uparrow,\downarrow\}$.

We write $u\prec^* v$ if
\[
\sum_{i=1}^k u_i^* \le \sum_{i=1}^k v_i^*, \quad k=1,\dots,N-1,
\qquad
\sum_{i=1}^N u_i^*=\sum_{i=1}^N v_i^*.
\]
\end{defn}

%The natural class of linear operators in this setting the doubly stochastic block matrices,
%where each $A_j$ is doubly stochastic. As in the two-block case, such operators are characterized by their preservation of $*$-majorization when each block is arranged in the prescribed order.
%This characterization is consistent with the two-block result given in Theorem~\ref{thm:block-ds-star}.

Existing results for the two-block setting extend naturally to the case of finitely many blocks. In particular, the appropriate class of linear operators is given by block diagonal matrices
\[
A=A_1\oplus \cdots \oplus A_k,
\]
where each \(A_j\) is doubly stochastic. As in the two-block case, these operators are characterized by their preservation of \( * \)-majorization when the coordinates within each block are arranged according to the prescribed order. Consequently, the analogues of Theorems~\ref{thm:block-ds-star} and~\ref{thm:block-equivalences} remain valid for finitely many blocks, with essentially the same proofs and notation adapted blockwise.

\section{Eigenvalue Inequalities and 
$*$-Majorization}

We begin by recalling two results on the variation of eigenvalues for Hermitian matrices: 
Cauchy’s Interlacing Theorem (see \cite[Corollary III.1.5]{Bhatia1997}) and Ky Fan’s inequality (see \cite{F0}). 
These results provide powerful tools for comparing the eigenvalues of a matrix and its principal submatrices 
will play a central role in our analysis.

\begin{thm}[Cauchy's Interlacing Theorem]
\label{2025-11-24-Cauchy}
Let $ A $ be a Hermitian matrix of size $n$ and let $M$ be a principal submatrix of $ A $ of size $m$. Their eigenvalues satisfy the following interlacing inequalities:
$$
\lambda_j^\downarrow(A) \ge \lambda_j^\downarrow(M) \ge \lambda_{n-m+j}^\downarrow(A), \text{ for all } j = 1, \dots, m.
$$
\end{thm}
The interlacing in the theorem can be stated equivalently as
$$
\lambda_{n-m+i}^\uparrow(A) \ge \lambda_{i}^\uparrow(M) \ge \lambda_{i}^\uparrow(A), \text{ for all } i = 1, \dots, m,
$$
or 
\begin{align}
\label{2025-11-24-interlac}
\lambda_{m-i+1}^\downarrow(A) \ge \lambda_{i}^\uparrow(M) \ge \lambda_{i}^\uparrow(A), \text{ for all } i = 1, \dots, m.
\end{align}

%\subsection*{Compression, projections, and Ky Fan’s variational principles}

Next, we recall Ky Fan’s variational principle. Let $A$ be a Hermitian matrix acting on $\mathbb{C}^r$, and let
$V\subset\mathbb{C}^r$ be a subspace of dimension $k$.  Let $P_V$ denote the orthogonal projection onto $V$. 
The {\it compression of $A$ to $V$} is the Hermitian matrix defined by
$A|_V := P_V A P_V$ acting on  $V,$ 
or equivalently,
\[
A|_V(x)=P_V(Ax), \qquad x\in V.
\]

If $\{u_1,\dots,u_k\}$ is an orthonormal basis of $V$, then the matrix of $A|_V$
with respect to this basis is
\[
(A|_V)_{ij}=\langle A u_j,u_i\rangle, \qquad i,j=1,\dots,k.
\]
Consequently,
\[
\operatorname{tr}(A|_V)
=
\sum_{i=1}^k \langle A u_i,u_i\rangle
=
\operatorname{tr}(P_VA).
\]
In particular, if $V=\mathrm{span}\{e_i:\, i\in\beta\}$ is a coordinate subspace of $\mathbb{C}^r$,
then the compression $A|_V$ is the principal submatrix
$A[\beta]$. We are now ready to state the Ky Fan’s variational principles, see \cite{F0}.

\begin{thm}
Let $A$ be a Hermitian matrix acting on $\mathbb{C}^r$. Then
\begin{equation} 
\label{KFmax}
\sum_{i=1}^k \lambda_i^\downarrow(A)
=
\max_{\substack{V\subset\mathbb{C}^r\\ \dim V=k}}
\operatorname{tr}(A|_V) \,\,\, \mbox{ and } \,\,\,  
\sum_{i=1}^k \lambda_i^\uparrow(A)
=
\min_{\substack{V\subset\mathbb{C}^r\\ \dim V=k}}
\operatorname{tr}(A|_V),
\end{equation}
for every $k=1,\dots,r$.
\end{thm}

%\medskip
%\noindent
%Since $\operatorname{tr}(A|_V)=\operatorname{tr}(P_VA)$, Ky Fan’s inequalities can
%be formulated interchangeably in terms of projections or compressions. This
%equivalence is fundamental when applying Ky Fan’s principles to compressions
%$A|_V$ that are unitarily similar to the principal submatrices of $A$. These results provide a bridge between spectral properties and variational characterizations. 
%They will allow us to derive structured inequalities for the eigenvalues of principal submatrices.

%The next result can be viewed as an extension of a connection between eigenvalue majorization and inequalities involving product of principal minors of a positive definite matrix. 

% The next result extends the potential relationships between eigenvalue majorization and Ky Fan–type variational inequalities to inequalities. 

We now present our main result, which establishes a family of $*$-majorization relations among the eigenvalue vectors of principal submatrices. These relations may be viewed as a spectral counterpart to Koteljanskii’s inequality.

\begin{thm}
\label{2025-11-24-thm}
Let $A$ be a positive definite matrix and $\alpha,\beta \subseteq N$.
Let
\begin{align*}
 x = \lambda(A[\alpha]) \in \mathbb{R}^n, \quad 
 y = \lambda(A[\beta]) \in \mathbb{R}^m, \quad 
 z = \lambda(A[\alpha \cup \beta]) \in \mathbb{R}^r, \quad 
 w = \lambda(A[\alpha \cap \beta]) \in \mathbb{R}^s.
\end{align*}
Then the following relationships hold:
\begin{enumerate}[(1)]
 \item \label{2026-01-27-maj1} $(x^{\downarrow},y^{\uparrow})\prec^*(z^{\downarrow},w^{\downarrow})$,
 \item \label{2026-01-27-maj2} $(x^{\downarrow},y^{\uparrow})\prec^*(z^{\downarrow},w^{\uparrow})$,
 \item \label{2026-01-27-maj3} $(x^{\downarrow},y^{\downarrow})\prec^*(z,w)^{\downarrow}$.
\end{enumerate}
\end{thm}

% \begin{thm}
% \label{2025-11-24-thm}
%  Let $A$ be a positive definite matrix and  $\alpha,\beta\subseteq N$.
%  Let \begin{align*}
%  x = \lambda(A[\alpha]) \in  \mathbb{R}^n, \,\,\, 
%  y = \lambda(A[\beta]) \in  \mathbb{R}^m, \,\,\, 
% z = \lambda(A[\alpha \cup \beta]) \in  \mathbb{R}^r, \,\,\, 
% w = \lambda(A[\alpha \cap \beta]) \in  \mathbb{R}^s.
% \end{align*}
% Then the following relationships hold:
% \begin{enumerate}[(1).]
%  \item \label{2026-01-27-maj1} $(x^{\downarrow},y^{\uparrow})\prec^*(z^{\downarrow},w^{\downarrow})$,
%  \item \label{2026-01-27-maj2} $(x^{\downarrow},y^{\uparrow})\prec^*(z^{\downarrow},w^{\uparrow})$,
%  \item \label{2026-01-27-maj3} $(x^{\downarrow},y^{\downarrow})\prec^*(z,w)^{\downarrow}$,
%  \item \label{2026-01-27-maj4} $(x^{\uparrow},y^{\uparrow})\ ^*\hspace{-0.1cm}\succ(z,w)^{\uparrow}$, %\footnote{What notation should we use $\succ^*$ or $^*\hspace{-0.1cm}\succ$?}  
%  \item \label{2026-01-27-maj5} $(x^{\uparrow},y^{\downarrow})\prec^*(z,w)^{\downarrow}$,
%  \item \label{2026-01-27-maj6} $(x^{\uparrow},y^{\uparrow})\prec^*(z,w)^{\downarrow}$,
%  \item \label{2026-01-27-maj8} $(x^{\downarrow},y^{\uparrow})\prec^*(z,w)^{\downarrow}$,
%  \item \label{2026-01-27-maj7} $(x^{\uparrow},y^{\downarrow})\ ^*\hspace{-0.1cm}\succ(z,w)^{\uparrow}.$
% \end{enumerate}
% \end{thm}

\begin{proof}
\begin{alignat*}{3}
X&:=\mathrm{span}\{e_i:\ i\in\alpha\},\qquad
&&Y&&:=\mathrm{span}\{e_i:\ i\in\beta\}, \\
Z&:=\mathrm{span}\{e_i:\ i\in\alpha\cup\beta\},\qquad
&&W&&:=\mathrm{span}\{e_i:\ i\in\alpha\cap\beta\}.
\end{alignat*}
Then
$\dim X=n$, $\dim Y=m$, $\dim Z=r$, $\dim W=s$,
with $n+m=r+s$.
Since compressions to coordinate subspaces are the corresponding principal
submatrices: $A[\alpha] =  A|_X$,
$A[\beta]= A|_Y$, $A[\alpha\cup\beta]= A|_Z$, $A[\alpha\cap\beta]= A|_W$, the vectors of eigenvalues of $A|_X,A|_Y,A|_Z,A|_W$ are $x,y,z,w$, respectively.

 \eqref{2026-01-27-maj1}. 
Let
\[
n = |\alpha|,\quad m = |\beta|,\quad r = |\alpha \cup \beta|,\quad s = |\alpha \cap \beta|.
\]
Then \(n+m = r+s\). We verify the  inequalities
\((x^\downarrow, y^\uparrow) \prec^{*} (z^\downarrow, w^\downarrow)\).

\medskip
\noindent
\textit{Case 1: Suppose \(1 \le k \le n\).}
Since \(A[\alpha]\) is a principal submatrix of \(A[\alpha \cup \beta]\), the inequalities of 
Cauchy's interlacing theorem give
\[
\sum_{i=1}^k x_i^\downarrow \le \sum_{i=1}^k z_i^\downarrow.
\]

\medskip
\noindent
\textit{Case 2: Let \(n < k \le r\).}
Write \(k = n+\ell\), where \(1 \le \ell \le r-n\).
Let \(\delta = \beta \setminus \alpha\) and \(\eta = \lambda(A[\delta])\),
so that \(|\delta| = r-n\). Then
\[
\sum_{i=1}^k (x^\downarrow, y^\uparrow)_i
=
\sum_{i=1}^n x_i + \sum_{j=1}^\ell y_j^\uparrow.
\]
Since \(A[\delta]\) is a principal submatrix of \(A[\beta]\), interlacing yields
\[
\sum_{j=1}^\ell y_j^\uparrow
\le
\sum_{j=1}^\ell \eta_j^\uparrow.
\]
Moreover, using the well-known pinching inequality (see \cite{Bhatia1997}) applied to \(A[\alpha \cup \beta]\)
with respect to the decomposition \(\alpha \cup \beta = \alpha \,{\cup}\, \delta\),
we have
$(x,\eta) \prec z.$ 
Thus
\[
\sum_{i=1}^{n+\ell} (x,\eta)_i^\downarrow
\le
\sum_{i=1}^{n+\ell} z_i^\downarrow.
\]
Since the sum of the largest \(n+\ell\) entries of the vector \((x,\eta)\) is at least
\(\sum_{i=1}^n x_i + \sum_{j=1}^\ell \eta_j^\uparrow\), it follows that
\[
\sum_{i=1}^n x_i + \sum_{j=1}^\ell y_j^\uparrow
\le
\sum_{i=1}^{n+\ell} z_i^\downarrow,
\]
which proves the desired inequality in this case.

\medskip
\noindent
\textit{Case 3: Lastly, assume that \(r < k \le r+s\).}
Write \(k = r+t\), where \(1 \le t \le s\), and let
\(\gamma = \alpha \cap \beta\). Then
$\beta = (\beta \setminus \alpha) \,{\cup}\, \gamma$,
$|\beta \setminus \alpha| = r-n$, and $|\gamma| = s$.
We have
\[
\sum_{i=1}^k (x^\downarrow, y^\uparrow)_i
=
\sum_{i=1}^n x_i + \sum_{j=1}^{(r-n)+t} y_j^\uparrow.
\]

By the Ky Fan minimum principle,
\[
\sum_{j=1}^{(r-n)+t} \lambda_j^\uparrow(A[\beta])
=
\min_{\dim V = (r-n)+t} \operatorname{tr}(P_V A[\beta]),
\]
and hence for any subspace \(V\) of dimension \((r-n)+t\),
\[
\sum_{j=1}^{(r-n)+t} y_j^\uparrow
\le
\operatorname{tr}(P_V A[\beta]).
\]

Choose
$ V = \mathbb{R}^{\beta \setminus \alpha} \oplus W,$ 
where \(W \subseteq \mathbb{R}^{\gamma}\) is spanned by eigenvectors of
\(A[\gamma]\) corresponding to its \(t\) largest eigenvalues. Then
$
\dim V = (r-n) + t.
$

With respect to the orthogonal decomposition
$\beta = (\beta \setminus \alpha) \,{\cup}\, \gamma,
$ 
the projection \(P_V\) is block diagonal, and therefore
\[
\operatorname{tr}(P_V A[\beta])
=
\operatorname{tr}(A[\beta \setminus \alpha])
+
\operatorname{tr}(P_W A[\gamma]).
\]
From the choice of \(W\),
\[
\operatorname{tr}(P_W A[\gamma])
=
\sum_{j=1}^t \lambda_j^\downarrow(A[\gamma])
=
\sum_{j=1}^t w_j^\downarrow.
\]
Thus
\[
\sum_{j=1}^{(r-n)+t} y_j^\uparrow
\le
\operatorname{tr}(A[\beta \setminus \alpha])
+
\sum_{j=1}^t w_j^\downarrow.
\]

Consequently,
\[
\sum_{i=1}^k (x^\downarrow, y^\uparrow)_i
\le
\operatorname{tr}(A[\alpha])
+
\operatorname{tr}(A[\beta \setminus \alpha])
+
\sum_{j=1}^t w_j^\downarrow.
\]

Since
$ \alpha \cup \beta = \alpha \,{\cup}\, (\beta \setminus \alpha),$  
we have
\[
\operatorname{tr}(A[\alpha]) + \operatorname{tr}(A[\beta \setminus \alpha])
=
\operatorname{tr}(A[\alpha \cup \beta])
=
\sum_{i=1}^r z_i^\downarrow.
\]
Therefore,
\[
\sum_{i=1}^k (x^\downarrow, y^\uparrow)_i
\le
\sum_{i=1}^r z_i^\downarrow + \sum_{j=1}^t w_j^\downarrow
=
\sum_{i=1}^k (z^\downarrow, w^\downarrow)_i.
\]
\medskip
Finally, for \(k = n+m = r+s\), both sides equal the total sum of eigenvalues,
and the equality follows from
\begin{equation}\label{2026-01-28-traces}
    \operatorname{tr}(A[\alpha]) + \operatorname{tr}(A[\beta])
=
\operatorname{tr}(A[\alpha \cup \beta]) + \operatorname{tr}(A[\alpha \cap \beta]).
\end{equation}
This completes the proof.

\eqref{2026-01-27-maj2}. Observe that the partial sums of $x^{\downarrow}$ are the same as in the previous case. Further, verifying the mixed sums involving $x^{\downarrow}$ and $y^{\uparrow}$, up to the first $r-n$ $y$'s, is also identical to the previous case.

For the remaining $s$ inequalities, consider the trace equality
$$
\sum_{i=1}^{n} x_i^\downarrow + \sum_{j=1}^{m} y_j^\uparrow = \sum_{i=1}^{r} z_i^\downarrow + \sum_{j=1}^{s} w_j^\uparrow.
$$
Additionally, since $ A[\alpha \cap \beta]$ (of size $s$) is a principal submatrix of $A[\beta]$ (of size $m$), Cauchy's interlacing theorem implies
$$
y_{m - s + j}^\uparrow \geq w_{j}^\uparrow \quad \text{for } j = 1, \dots, s.
$$

Fix $ p \in \{1, \dots, s\} $. Subtracting the $p$ largest terms from both sides of the trace equality, namely, $y_{m - p + 1}^\uparrow, \dots, y_m^\uparrow$ on the left and $w_{s - p + 1}^\uparrow, \dots, w_s^\uparrow$ on the right and applying the inequalities above, we obtain
$$
\sum_{i=1}^{n} x_i^\downarrow + \sum_{j=1}^{m - p} y_j^\uparrow \leq \sum_{i=1}^{r} z_i^\downarrow + \sum_{j=1}^{s - p} w_j^\uparrow,
$$
for all $ p = 1, \dots, s $.

\eqref{2026-01-27-maj3}. Let $b:=(z,w)^\downarrow$
denote the vector obtained from $(z,w)$ by rearranging all entries in non-increasing order.
Since $A|_U$ is a compression of $A|_S$, Cauchy’s interlacing theorem gives
\[
x_j^\downarrow \le z_j^\downarrow,
\qquad j=1,\dots,n.
\]
Hence,
\[
\sum_{j=1}^k x_j^\downarrow \le \sum_{j=1}^k z_j^\downarrow \le \sum_{j=1}^k b_j,
\qquad k=1,\dots,n,
\]
since $b=(z,w)^\downarrow$ consists of the largest entries among those of $z$ and
$w$. 
Next, the trace equality gives
\[
\sum_{i=1}^n x_i^\downarrow + \sum_{j=1}^m y_j^\downarrow
=
\sum_{i=1}^{r+s} (z,w)_i^\downarrow.
\]
% Fix $\ell\in\{1,\dots,m\}$ and set
% \[
% p:=m-\ell.
% \]
% By Condition (3), the desired inequality
% \[
% \sum_{i=1}^n x_i^\downarrow+\sum_{j=1}^\ell y_j^\downarrow
% \le
% \sum_{i=1}^{n+\ell}(z,w)_i^\downarrow
% \]
% is equivalent to
% \[
% \sum_{j=\ell+1}^m y_j^\downarrow
% \ge
% \sum_{i=n+\ell+1}^{r+s}(z,w)_i^\downarrow.
% \]
% Since $p=m-\ell$ and $r+s=n+m$, this is the same as
% \[
% \sum_{j=1}^{p} y_j^\uparrow
% \ge
% \sum_{i=1}^{p} (z,w)_i^\uparrow.
% \]

% Now apply Ky Fan's minimum principle to $A|_V$. There exists a
% $p$-dimensional subspace $F\subseteq V$ such that
% \[
% \sum_{j=1}^{p} y_j^\uparrow=\operatorname{tr}(A|_F).
% \]
% Let
% \[
% F_1:=F\cap W,\qquad t:=\dim F_1.
% \]
% Choose $F_2$ to be the orthogonal complement of $F_1$ in $F$. Then
% \[
% F=F_1\oplus F_2,\qquad \dim F_2=p-t,\qquad F_2\subseteq F\subseteq V\subseteq S.
% \]
% Hence
% \[
% \operatorname{tr}(A|_F)=\operatorname{tr}(A|_{F_1})+\operatorname{tr}(A|_{F_2}).
% \]

For a fixed $\ell \in \{1,\dots,m\}$, we need to show the inequality
\[
\sum_{i=1}^n x_i^\downarrow + \sum_{j=1}^\ell y_j^\downarrow
\;\le\;
\sum_{i=1}^{n+\ell} (z,w)_i^\downarrow,
\]
which is equivalent (by subtracting both sides from the trace equality) to
\[
\sum_{j=\ell+1}^m y_j^\downarrow
\;\ge\;
\sum_{i=n+\ell+1}^{r+s} (z,w)_i^\downarrow.
\]

Let $p := m - \ell$ and since $r+s = n+m$, both sides consist of sums of the
$p$ smallest entries of the respective sequences. Therefore,
\[
\sum_{j=\ell+1}^m y_j^\downarrow
=
\sum_{j=1}^{p} y_j^\uparrow,
\qquad
\sum_{i=n+\ell+1}^{r+s} (z,w)_i^\downarrow
=
\sum_{i=1}^{p} (z,w)_i^\uparrow,
\]
and the above inequality becomes
\begin{align}
\label{2026-04-20-ineq}
\sum_{j=1}^{p} y_j^\uparrow
\;\ge\;
\sum_{i=1}^{p} (z,w)_i^\uparrow.
\end{align}

We now focus on proving the last inequality.
By the Ky Fan's minimum principle, there exists a $p$-dimensional subspace $F \subseteq V$, such that
\[
\sum_{j=1}^{p} y_j^\uparrow
=
\operatorname{tr}(A|_F).
\]

Define
$F_1 := F \cap W$ and $t := \dim F_1.$
Let $F_2$ be the orthogonal complement of $F_1$ relative to $F$, that is,
\[
F = F_1 \oplus F_2,
\qquad
\dim F_2 = p - t,
\]
with $F_2 \subseteq F \subseteq V \subseteq S.$
The trace of the compression splits as 
\[
\operatorname{tr}(A|_F)
=
\operatorname{tr}(A|_{F_1})
+
\operatorname{tr}(A|_{F_2}).
\]
%and we have the following equations
%$P_F = P_{F_1} + P_{F_2}$ and $P_{F_1}P_{F_2}=P_{F_2}P_{F_1}=0.$ 
%Therefore,
%\[
%P_F A P_F
%=
%P_{F_1} A P_{F_1}
%+
%P_{F_1} A P_{F_2}
%+
%P_{F_2} A P_{F_1}
%+
%P_{F_2} A P_{F_2}.
%\]
%Taking traces and using the cyclicity property of the trace function, we have $\operatorname{tr}(P_{F_1} A P_{F_2})
%=
%\operatorname{tr}(P_{F_2} P_{F_1} A)
%=
%0$ 
%and similarly
%$\operatorname{tr}(P_{F_2} A P_{F_1})=0.$ 
%Hence,
%\[
%\operatorname{tr}(P_F A|_F)
%=
%\operatorname{tr}(P_{F_1} A|_{F_1})
%+
%\operatorname{tr}(P_{F_2} A|_{F_2}).
%\]
%
Since $F_1 \subseteq W$ and $\dim F_1 = t$, Ky Fan's minimum principle applied to $A|_W$ implies
\[
\operatorname{tr}(A|_{F_1})
\;\ge\;
\sum_{j=1}^t w_j^\uparrow.
\]

Similarly, since $F_2 \subseteq S$ and $\dim F_2 = p - t$, applying Ky Fan's principle to $A|_S$ implies
\[
\operatorname{tr}(A|_{F_2})
\;\ge\;
\sum_{i=1}^{p-t} z_i^\uparrow.
\]

%Using the orthogonal decomposition $F = F_1 \oplus F_2$, we have
%\[
%\operatorname{tr}(P_F A|_F)
%=
%\operatorname{tr}(P_{F_1} A|_{F_1})
%+
%\operatorname{tr}(P_{F_2} A|_{F_2}).
%\]
Therefore,
\[
\sum_{j=1}^{p} y_j^\uparrow
=
\operatorname{tr}(A|_F)
\;\ge\;
\sum_{i=1}^{p-t} z_i^\uparrow
+
\sum_{j=1}^t w_j^\uparrow.
\]

Finally, since the right-hand side is the sum of $p$ elements drawn from the combined sequence $(z,w)$, it is bounded below by the sum of the $p$ smallest elements of $(z,w)$:
\[
\sum_{j=1}^{p} y_j^\uparrow
\;\ge\; \sum_{i=1}^{p-t} z_i^\uparrow
+
\sum_{j=1}^t w_j^\uparrow
\;\ge\;
\sum_{i=1}^{p} (z,w)_i^\uparrow,
\]
proving \eqref{2026-04-20-ineq}.
\end{proof}

\begin{cor}
Under the assumptions of Theorem \ref{2025-11-24-thm}, the following also hold:
\begin{enumerate}[(1)]
 \item \label{2026-01-27-maj4} $(x^{\uparrow},y^{\uparrow})\ ^*\hspace{-0.1cm}\succ(z,w)^{\uparrow}$,
 \item \label{2026-01-27-maj5} $(x^{\uparrow},y^{\downarrow})\prec^*(z,w)^{\downarrow}$,
 \item \label{2026-01-27-maj6} $(x^{\uparrow},y^{\uparrow})\prec^*(z,w)^{\downarrow}$,
 \item \label{2026-01-27-maj8} $(x^{\downarrow},y^{\uparrow})\prec^*(z,w)^{\downarrow}$,
 \item \label{2026-01-27-maj7} $(x^{\uparrow},y^{\downarrow})\ ^*\hspace{-0.1cm}\succ(z,w)^{\uparrow}$.
\end{enumerate}
\end{cor}
\begin{proof}
    \eqref{2026-01-27-maj4}. This follows by applying item~\eqref{2026-01-27-maj3} of Theorem~\ref{2025-11-24-thm} to $-A$.

    \eqref{2026-01-27-maj5}. Set $b:=(z,w)^\downarrow$.
For every $k=1,\dots,n$, the sum of the $k$ smallest entries of $x$ is bounded
above by the sum of the $k$ largest entries of $x$, that is,
\[
\sum_{j=1}^k x_j^{\uparrow}\le \sum_{j=1}^k x_j^{\downarrow} \le \sum_{j=1}^k b_j,
\]
where the last inequality holds since we already know  $(x^{\downarrow},y^{\downarrow})\prec^*(z,w)^{\downarrow}$.

The remaining inequalities in the definition of
$(x^{\uparrow},y^{\downarrow})\prec^*(z,w)^{\downarrow}$ are of the form
\[
\sum_{i=1}^{n} x_i^\uparrow+\sum_{j=1}^{\ell} y_j^{\downarrow}\le \sum_{i=1}^{n+\ell} b_i,
\qquad \ell=1,\dots,m.
\]
These are proved precisely as in the case
$(x^{\downarrow},y^{\downarrow})\prec^*(z,w)^{\downarrow}$, since they depend
only on the full sum $\sum_{i=1}^n x_i=\operatorname{tr}(A[\alpha])$ and not on
a particular rearrangement of the entries in $x$. The last inequality is actually an equality due to \eqref{2026-01-28-traces}.

\eqref{2026-01-27-maj6}. Set $b:=(z,w)^{\downarrow}$ and recall that
$(x^{\downarrow},y^{\downarrow})\prec^*(z,w)^{\downarrow}$,
that is, in particular,
\begin{align}
\sum_{j=1}^k x_j^{\downarrow} &\le \sum_{j=1}^k b_j, \,\,\, \mbox{ for } k=1,\dots,n, \label{eq:known1}\\
\sum_{j=1}^n x_j^\downarrow+\sum_{j=1}^{\ell} y_j^{\downarrow}
&\le \sum_{j=1}^{n+\ell} b_j,\,\,\, \mbox{ for } \ell=1,\dots,m, \label{eq:known2}
\end{align}
together with the corresponding trace condition.

For every $k=1,\dots,n$, the sum of the $k$ smallest entries of a vector is at
most the sum of its $k$ largest entries, hence
\[
\sum_{j=1}^k x_j^{\uparrow}\le \sum_{j=1}^k x_j^{\downarrow} 
\le \sum_{j=1}^k b_j,
\qquad k=1,\dots,n.
\]
Combining this inequality with (\ref{eq:known2}), we have the following, for every $\ell=1,\dots,m$,
\[
\sum_{j=1}^n x_j^\uparrow+\sum_{j=1}^{\ell} y_j^{\uparrow}
\le
\sum_{j=1}^n x_j^\uparrow+\sum_{j=1}^{\ell} y_j^{\downarrow}
\le
\sum_{j=1}^{n+\ell} b_j.
\]
The last inequality holds with equality again.

\eqref{2026-01-27-maj8}. Set $b:=(z,w)^\downarrow$ and recall 
$(x^{\downarrow},y^{\downarrow})\prec^*(z,w)^{\downarrow}.$
From it, for every
$k=1,\dots,n$,
\[
\sum_{j=1}^k x_j^{\downarrow}\le \sum_{j=1}^k b_j.
\]
Fix $\ell\in\{1,\dots,m\}$. Since the sum of the $\ell$ smallest entries of a
vector is at most the sum of its $\ell$ largest entries, we have
\[
\sum_{j=1}^{\ell} y_j^{\uparrow}\le \sum_{j=1}^{\ell} y_j^{\downarrow}.
\]
Therefore,
\[
\sum_{j=1}^{n} x_j^\downarrow + \sum_{j=1}^{\ell} y_j^{\uparrow}
\le
\sum_{j=1}^{n} x_j^\downarrow + \sum_{j=1}^{\ell} y_j^{\downarrow}
\le
\sum_{j=1}^{n+\ell} b_j,
\]
where in the last inequality, we used $(x^{\downarrow},y^{\downarrow})\prec^*(z,w)^{\downarrow}$.
The last inequality holds with equality again.

\eqref{2026-01-27-maj7}. Set $c:=(z,w)^{\uparrow}$. For any real vector $v$, we have
$(-v)^{\downarrow}=-v^{\uparrow}$ and $(-v)^{\uparrow}=-v^{\downarrow}$.
In particular, $(-z,-w)^{\downarrow}=-(z,w)^{\uparrow}=-c.$ Thus, the relation $(x^{\uparrow},y^{\downarrow})\;{}^{*}\!\succ\;(z,w)^{\uparrow}$
is equivalent to 
$$
((-x)^{\downarrow},(-y)^{\uparrow}) =(-x^{\uparrow},-y^{\downarrow})=-(x^{\uparrow},y^{\downarrow})\; \prec^* \; - (z,w)^{\uparrow} = (-z,-w)^{\downarrow}
$$
This is precisely the $*$-majorization relationship \eqref{2026-01-27-maj8}, applied to $-A$.
\end{proof}

\section*{Acknowledgments}
S.M.\ Fallat is supported in part by an NSERC Discovery Research Grant, Application No.: RGPIN-2025-05272.
 The work of PIMS Postdoctoral Fellow S.\ Mondal leading to this publication was supported in part by the Pacific Institute for the Mathematical Sciences.  H.\ Sendov is supported in part by an NSERC Discovery Research Grant, Application No.: RGPIN-2020-06425.


\begin{thebibliography}{10}

 \bibitem{Bhatia1997}
 Bhatia R, {\it Matrix analysis}, New York: Springer-Verlag; 1997.
 



% \bibitem{FJb}
% S.M. Fallat and C.R. Johnson, 
% {\it Totally Nonnegative Matrices,} 
% Princeton University Press, Princeton, USA, 2011.

\bibitem{F0}
K Fan, \textit{On a Theorem of Weyl Concerning Eigenvalues of Linear Transformations I*}, Proc. Natl. Acad. Sci. \textbf{35 (11)} (1949), 652-655.



\bibitem{Fis}
E. Fischer, \textit{\"{U}ber den Hadamard'schen  determinantensatz},
Archiv. Math. Physik (3), \textbf{13} (1908), 32--40.

\bibitem{FischerHolbrook1977}
P.~Fischer and J.~A.~R.~Holbrook,
\newblock Matrices doubly stochastic by blocks,
\newblock {\it Canadian Journal of Mathematics}, 29(3) (1977), 559--577.

\bibitem{Had}
J. Hadamard, \textit{R\'{e}soultion d'une question relative aux
d\'{e}terminants}, Bull. Sci. Math. S\'{e}r. 2, \textbf{17} (1893),
240--246.
\bibitem{Hardy1929}
G. H. Hardy, J. E. Littlewood, and G. Pólya,
\textit{Some simple inequalities satisfied by convex functions},
Messenger of Mathematics, \textbf{58} (1929), 145--152.

\bibitem{HardyLittlewoodPolya1952}
G.~H.~Hardy, J.~E.~Littlewood, and G.~P\'olya,
\newblock {\it Inequalities}, 2nd ed.,
\newblock Cambridge University Press, Cambridge, 1952.



 \bibitem{Ko1}
D.M. Koteljanskii, 
The Theory of Nonnegative and Oscillating Matrices (Russian)
{\it Ukrain. Mat. \u{Z}.} {\bf 2}:94-101 (1950); 
English transl.: {\it Translations of the AMS,} Series 2 {\bf 27} (1963), 1-8.

\bibitem{Ko2}
D.M. Koteljanskii, A Property of Sign-Symmetric Matrices (Russian),
{\it Mat. Nauk (N.S.)} {\bf 8}:163-167 (1953); 
English transl.: {\it Translations of the AMS,} Series 2 {\bf 27} (1963), 19-24.

\bibitem{Marshall2011}
A. W. Marshall, I. Olkin, and B. C. Arnold,
\textit{Inequalities: Theory of Majorization and Its Applications},
2nd Edition, Springer-Verlag, 2011.

\bibitem{Niculescu2006}
C. P. Niculescu and L.-E. Persson,
\textit{Convex Functions and Their Applications: A Contemporary Approach},
CMS Books in Mathematics, Vol. 23, Springer-Verlag, New York, 2006.

\bibitem{Nielsen2000}
M. A. Nielsen and I. L. Chuang,
\textit{Quantum Computation and Quantum Information},
Cambridge University Press, 2000.
\textbf{30} (1958), 175--210.




\bibitem{Phelps2001}
R. R. Phelps,
\textit{Lectures on Choquet’s Theorem},
2nd Edition, Lecture Notes in Mathematics, No. 1757,
Springer-Verlag, Berlin, 2001.

\bibitem{Rado1952}
R.~Rado,
\newblock An inequality,
\newblock {\it J. London Math. Soc.} 27 (1952), 1--6.


\bibitem{Ryff1965}
J. V. Ryff,  {Orbits of $L^1$-functions under doubly stochastic transformations}, 
{\it Trans. Amer. Math. Soc.}, \textbf{117} (1965), 92--100.

\bibitem{SendovYuan2026}
H.~Sendov and M.~Yuan, Aggregate Bounds on the Eigenvalues of the Principal Submatrices of a Hermitian Matrix and Majorization Relations,
\href{https://arxiv.org/abs/2601.16320}{arXiv:2601.16320}, 2026.




 \end{thebibliography}
\end{document}